\documentclass[a4paper,12pt]{article}
\usepackage{amsfonts}
\usepackage{amsthm}
\usepackage{verbatim}
\usepackage{amsmath, amssymb}
\numberwithin {equation}{section}
\newtheorem{theorem}{Theorem}[section]
\newtheorem{lemma}[theorem]{Lemma}

\newtheorem{example}[theorem]{Example}
\newtheorem{definition}[theorem]{Definition}

\newtheorem{remark}[theorem]{Remark}

\begin{document}
\setcounter{section}{0}

\begin{center}
\Large{\textbf{On Adjoint Additive Processes.}}
\end{center}
\begin{center}
\large{Kristian P. Evans \ \ \ Niels Jacob}
\end{center}
\begin{center}
Mathematics Department, \\ Swansea University, \\ Singleton Park, \\  Swansea, SA2 8PP.\\ U.K.
\end{center}

\section*{Abstract}
Starting with an additive process $(Y_t)_{t\geq0}$, it is in certain cases possible to construct an adjoint process $(X_t)_{t\geq0}$ which is itself additive. Moreover, assuming that the transition densities of $(Y_t)_{t\geq0}$ are controlled by a natural pair of metrics $\mathrm{d}_{\psi,t}$ and $\delta_{\psi,t}$, we can prove that the transition densities of $(X_t)_{t\geq0}$ are controlled by the metrics $\delta_{\psi,1/t}$ replacing $\mathrm{d}_{\psi,t}$ and $\mathrm{d}_{\psi,1/t}$ replacing $\delta_{\psi,t}$.

\vskip10pt\noindent \textbf{AMS Subject Classification:} 60J30, 60J35, 60E07, 60E10, 47D03, 47D06
\newline \textbf{Abbreviated Title:} Adjoint Additive Processes
\newline \textbf{Keywords:} Additive processes; L\'evy processes, adjoint densities, transition functions, metric measure spaces.
\section*{Introduction}
The origin of this investigation is the paper \cite{J2} where it was suggested to understand the transition density $p_t(x)$ of a symmetric L\'evy process $(Y_t)_{t\geq 0}$ with characteristic exponent $\psi$ in terms of two in general $t$-dependent metrics $d_{\psi,t}=\sqrt{t}\,\mathrm{d}_\psi$, where $d_\psi(\xi,\eta)=\psi^\frac12(\xi-\eta)$, and $\delta_{\psi,t}$, i.e.,
\begin{equation}\label{EQ0.1}
p_t(x-y)=p_t(0)e^{-\delta_{\psi,t}^2(x,y)}
\end{equation}
and
\begin{equation}\label{EQ0.2}
p_t(0)=(2\pi)^{-n}\int_0^\infty\lambda^{(n)}(B^{\mathrm{d}_\psi}(0,\sqrt{r/t}))e^{-r}\,\mathrm{d}r.
\end{equation}
The term \eqref{EQ0.2} has already been considered in \cite{K1}. While the metric $\mathrm{d}_{\psi,t}$ is, under mild conditions, always at our disposal, the existence of $\delta_{\psi,t}$ is in general an open problem. Examples in \cite{J2} suggest that in some cases $x\mapsto\delta^2_{\psi,t}(x,0)$ for $t>0$ fixed is itself the characteristic exponent of a L\'evy process, i.e. a continuous negative definite function, and that $(t,x)\mapsto\delta^2_{\psi,1/t}(x,0)$ is the characteristic exponent of an additive process $(X_t)_{t\geq 0}$. An example is of course Brownian motion, a further one is the Cauchy process $(Y_t)_{t\geq 0}$ where the corresponding additive process $(X_t)_{t\geq 0}$ is the Laplace process. In \cite{B4}, the relations between the transition densities of $(Y_t)_{t\geq0}$ and $(X_t)_{t\geq0}$ were studied in more detail when $(Y_t)_{t\geq0}$ is a L\'evy process and when $(X_t)_{t\geq0}$ exists, i.e. $x\mapsto\delta_{\psi,t}^2(x,0)$ is a continuous negative definite function and $\delta^2_{\psi,1/t}(x,0)$ is the characteristic exponent of an additive process. A natural question is whether it is possible to already start with an additive process $(Y_t)_{t\geq0}$ with generator $-q(t,D)$, where $q(t,D)$ is a pseudo-differential operator with symbol $q(t,\xi)$, and for $t>0$ fixed $\xi\mapsto q(t,\xi)$ is the characteristic exponent of a L\'evy process, and to obtain a new additive process $(X_t)_{t\geq0}$ similar to the construction when starting with a L\'evy process. Additive processes can be traced back to P. L\'evy and this notion was  further clarified by K. It\^o as well as A.V. Skorohod, we refer to the notes in \cite{S1}.
\vskip10pt\noindent While pursuing these ideas, we learned about the work initiated by T. Lewis \cite{L2} who was (to the best of our knowledge) the first to consider probability distributions which are characteristic functions themselves. Such distributions he called adjoint. In the monograph \cite{L1}, adjoint distributions were discussed in more detail. Thus in light of these investigations and the discussion in \cite{J2} and \cite{B4}, we consider our paper as a further step to understand adjoint additive processes with densities $\Phi_t$. Here we call $(X_t)_{t\geq0}$ adjoint to $(Y_t)_{t\geq0}$ if there exists a mapping $j:(0,\infty)\to(0,\infty)$ such that for all $t\in(0,\infty)$ we have
\begin{equation}
\hat{p}_t=\Phi_{j(t)},
\end{equation}
where $\hat{p}_t$ is the Fourier transform of $p_t$. Often $j(t)=\frac1t$ will be a suitable choice. 
\vskip10pt\noindent Our approach is essentially an analytic one, namely to construct, with the help of  $p_t$, a symbol of an operator $A(t,D)$ which admits a fundamental solution such that this fundamental solution allows us to construct the transition densities $\Phi_t$ of an additive process. Given $p_t$, with $\sigma_t(\xi):=\frac{p_{1/t}(\xi)}{p_{1/t}(0)}$ we have to take $A(t,\xi)=-\frac{\partial}{\partial t}\ln\sigma_t(\xi)$. Beside some more or less standard technical assumptions we need the crucial, but restrictive \textbf{Basic Assumption I}: $\xi\mapsto A(t,\xi)$ is a continuous negative definite function, i.e. for fixed $t>0$ it has a L\'evy-Khintchine representation.
\vskip10pt\noindent
We then turn to the question of understanding the structure of transition densities, and for this we add \textbf{Basic Assumption II}: $\mathrm{d}_\psi(\xi,\eta):=\sqrt{\psi(\xi-\eta)}$ is a metric on $\mathbb{R}^n$ generating the Euclidean topology and $(\mathbb{R}^n,\mathrm{d}_\psi,\lambda^{(n)})$ is a metric measure space having the volume doubling property. Under these two basic assumptions and, as previously mentioned, some standard assumptions on the symbol $q(t,\xi)$ of the generator of the additive process $(Y_t)_{t\geq0}$ we start with, we can show that $(Y_t)_{t\geq0}$ admits an adjoint process $(X_t)_{t\geq 0}$. In addition, with $Q_{t,0}(\xi)=\int_0^tq(\tau,\xi)\,\mathrm{d}\tau$ and $\mathrm{d}_{Q_{t,0}}(\xi,\eta)=Q_{t,0}^\frac12(\xi-\eta)$, we have for the transition density $p_t(x-y)$ of $Y_t$
\begin{equation}
p_t(x-y)=(2\pi)^{-n}\int_{\mathbb{R}^n}\lambda^{(n)}(B^{\mathrm{d}_{Q_t,0}}(0,\sqrt{r}))e^{-r}\,\mathrm{d}r\,e^{-\delta^2_{Q_{t,0}}(x,y)}
\end{equation}
and for the transition density $\Phi_t$ of $X_t$ we find
\begin{equation}
\Phi_t(x-y)=(2\pi)^{-n}\int_{\mathbb{R}^n}\lambda^{(n)}(B^{\delta_{Q_{1/t,0}}}(0,\sqrt{r}))e^{-r}\,\mathrm{d}r\,e^{-d^2_{Q_{1/t,0}}(x,y)}.
\end{equation}
Of importance, of course, are examples and they are provided with the help of the symbols $q_1(t,\xi)=h_1(t)|\xi|^2$, $q_2(t,\xi)=h_2(t)|\xi|$ and $q_3(t,\xi)=h_3(t)\ln\cosh\xi$ (here we require $\xi\in\mathbb{R}$). Clearly certain combinations such as direct sums lead to more examples. As indicated in \cite{J2}, in particular Theorem 7.1, subordination in the sense of Bochner, see \cite{S2} for the general theory, shall lead to further examples. Readers with an interest in state of the art results of the theory of Markov processes related to pseudo-differential operators are referred to Schilling et al. \cite{B3} as well as to F. K\"uhn \cite{K2} and the forthcoming survey \cite{J3}. Whether it is possible to extend our considerations to the classes of processes constructed in \cite{B2} using the symbolic calculus of Hoh \cite{H} and in \cite{Z} using the ideas of \cite{J1} with the help of $x$ and $t$ dependent negative definite symbols remains an open question. 

\section{Adjoint Processes}

Let $(\Omega,\mathcal{A},P^x,(X_t)_{t\geq 0})_{x\in\mathbb{R}^n}$ be a stochastic process (adapted to a suitable filtration). Following K. Sato \cite{S1}, we call $(X_t)_{t\geq0}$ an \textbf{additive process in law} if $(X_t)_{t\geq 0}$ has independent increments and if it is stochastically continuous. If, in addition, the increments are also stationary, we call $(X_t)_{t\geq0}$ a \textbf{L\'evy process}. For the distribution $\gamma_{t,s}$ of the increments $X_t-X_s$, $0\leq s<t$, of an additive process, the following conditions are satisfied:
\begin{align}
&\gamma_{s,s}=\epsilon_0, \quad 0\leq s;\\
&\gamma_{t,r}\ast\gamma_{r,s}=\gamma_{t,s}, \quad 0\leq s\leq r\leq t;\\
&\gamma_{t,s}\to\epsilon_0 \ \text{ weakly for }\ s\to t, s<t;\\
&\gamma_{t,s}\to\epsilon_0 \ \text{ weakly for }\  t\to s, s<t.
\end{align} 

\noindent In the case of a L\'evy process we have $\gamma_{t,s}=\mu_{t-s}$ and $(\mu_t)_{t\geq0}$ is a convolution semi-group of probability measures on $\mathbb{R}^n$, i.e.,
\begin{align*}
&\mu_0=\epsilon_0\\
&\mu_t\ast\mu_s=\mu_{t+s}\\
&\mu_t\to\epsilon_0 \ \text{ weakly as } \ t\to 0.
\end{align*}
A continuous function $\psi:\mathbb{R}^n\to\mathbb{C}$ is called a \textbf{continuous negative definite function} if $\psi(0)\geq0$ and if for all $t>0$ the function $\xi\mapsto e^{-t\psi(\xi)}$ is positive definite in the sense of Bochner. Given a convolution semi-group of probability measures on $\mathbb{R}^n$ then there exists a unique continuous negative definite function $\psi:\mathbb{R}^n\to\mathbb{C}$ such that
\begin{equation}\label{EQ1.5}
\hat{\mu}_t(\xi)=(2\pi)^{-\frac{n}{2}}\int_{\mathbb{R}^n}e^{-ix\cdot\xi}\mu(\mathrm{d}x)=(2\pi)^{-\frac{n}{2}}e^{-t\psi(\xi)}
\end{equation}
holds.
\vskip5pt\noindent A remark about the normalisation of the Fourier transform is in order. Our choice is the common one in the theory of pseudo-differential operators and it has the property that the constant in Plancherel's theorem is equal to 1, i.e. we have $\|\hat{u}\|_0=\|u\|$ for $u\in L^2(\mathbb{R}^n)$ where $\|u\|_0$ denotes the $L^2$-norm of $u$. This is for many of our calculations rather convenient. Probabilists would prefer a different normalisation, either 
$$
\hat{\mu}_t(\xi)=(2\pi)^{-n}\int_{\mathbb{R}^n}e^{-ix\cdot\xi}\mu(\mathrm{d}x)
$$
or 
$$
\hat{\mu}_t(\xi)=\int_{\mathbb{R}^n}e^{ix\cdot\xi}\mu(\mathrm{d}x).
$$
Obviously the main results will be independent of this choice. In our normalisation the convolution theorem reads as 
$$
(\mu_t\ast\mu_s)^{\wedge}(\xi)=(2\pi)^{\frac{n}{2}}\hat{\mu}_t(\xi)\hat{\mu}_s(\xi)
$$
and the inverse Fourier transform is given by
$$
(F^{-1}u)(x)=(2\pi)^{-\frac{n}{2}}\int_{\mathbb{R}^n}e^{ix\cdot\xi}u(\xi)\,\mathrm{d}\xi.
$$
If $\mu_t=p_t(\cdot)\lambda^{(n)}$ then we have of course $\hat{\mu}_t=\hat{p}_t$ and from \eqref{EQ1.5} it follows that 
\begin{align*}
p_t(x)&=F^{-1}(\hat{\mu}_t)(x)=F^{-1}((2\pi)^{-\frac{n}{2}}e^{-t\psi(\cdot)})(x)\\
&=(2\pi)^{-n}\int_{\mathbb{R}^n}e^{ix\cdot\xi}e^{-t\psi(\xi)}\mathrm{d}\xi.
\end{align*}
 Here and in the following, $\hat{\mu}$ denotes the Fourier transform of $\mu$ and $F^{-1}u$ is the inverse Fourier transform of $u$. If the continuous negative definite function $\psi$ is real-valued, the measures $\mu_t$ are symmetric and in this note we are only interested in the symmetric case. Moreover, we do not allow a killing or diffusion part and therefore the L\'evy-Khintchine representation of $\psi$ is given by
\begin{equation}
\psi(\xi)=\int_{\mathbb{R}\setminus\{0\}}(1-\cos(y\cdot\xi))\nu(\mathrm{d}y)
\end{equation}
with L\'evy measure $\nu$.
\vskip5pt\noindent A probability measure $\mu$ on $\mathbb{R}^n$ is called \textbf{infinitely divisible} if for every $k\in\mathbb{N}$ there exists a probability measure $\mu_k$ on $\mathbb{R}^n$ such that
\begin{equation}
\mu=\mu_k\ast\cdots\ast\mu_k\quad \text{($k$-terms)}.
\end{equation}
It is known, see \cite{B1}, that every infinitely divisible measure $\mu$ can be embedded into a convolution semigroup $(\mu_t)_{t\geq 0}$, $\mu_1=\mu$. 
\vskip5pt\noindent Following T. Lewis \cite{L2}, we call a probability distribution $p$ on $\mathbb{R}^n$ \textbf{adjoint} to a probability distribution $\Phi$ if
\begin{equation}
\hat{p}=\Phi.
\end{equation}
We call $p$ \textbf{self-adjoint} if
\begin{equation}
\hat{p}=p,
\end{equation}
i.e. if $p$ is a fixed point of the Fourier transform. Note that at this point the choice of the normalisation of the Fourier transform must be taken into account. Examples of adjoint distributions are, see \cite{L1},
\begin{align*}
p(x)&=\frac{2x}{\pi^2\sinh x}, \quad \Phi(x)=\frac{\pi}{4\cosh\tfrac{\pi x}{2}},\\
p(x)&=\frac1\pi\left(\frac{\sin x}{x}\right)^2, \quad \Phi(x)=\frac12\max(1-\frac{|x|}{2},0),
\end{align*}
and in addition to the normal distribution we find that
\begin{align}
p(x)&=\frac{1}{\sqrt{2\pi}\cosh(\sqrt{\tfrac\pi2}x)},\\
p(x)&=\frac{1}{\sqrt{2\pi}}\frac{\cos(\sqrt{\tfrac\pi2}x)}{\cosh(\sqrt{\pi}x)}
\end{align}
or
\begin{equation}
p_k(x)=C_k(H_{4 k}(\sqrt2 x)-m_{4k})e^{\frac{x^2}{2}},
\end{equation}
where $H_l$ is the $l^\text{th}$ Hermite polynomial, are self-adjoint distributions.
\vskip5pt\noindent If a distribution $p$ has an adjoint distribution $\Phi$ which is infinitely divisible the corresponding convolution semi-group $(\Phi_t)_{t\geq0}$ give rise to a L\'evy process. We call two stochastic processes with distribution $(p_t)_{t\geq0}$ and $(\Phi_t)_{t\geq0}$ \textbf{adjoint processes} if for a bijective mapping $j:(0,\infty)\to(0,\infty)$ we have
$$
\hat{p}_t=\Phi_{j(t)},
$$
where we will often use $j(t)=\frac1t$. One aim of the paper is to study this notion for L\'evy and additive processes.

\section{Some Additive Processes}
In the following, let $q:[0,\infty)\times\mathbb{R}^n\to\mathbb{R}$ be a continuous function such that for every $t\geq 0$ the function $q(t,\cdot):\mathbb{R}^n\to\mathbb{R}$ is a continuous negative definite function. It follows that $q(t,\xi)\geq0$ and for $0\leq s<t$
\begin{equation}
\xi\mapsto\int_s^tq(\tau,\xi)\,\mathrm{d}t
\end{equation}
is a continuous negative definite function too. We assume, in addition, that for a fixed continuous negative definite function $\psi:\mathbb{R}^n\to\mathbb{R}$ we have $\lim_{|\xi|\to\infty}\psi(\xi)=\infty$, $e^{-t\psi}\in L^1(\mathbb{R}^n),$ and for $0<\kappa_0<\kappa$
\begin{equation}\label{EQ2.2}
\kappa_o\nu_0(A)\leq\nu(t,A)\leq\kappa_1\nu_0(A), \ \ A\in\mathcal{B}^{(n)}(\mathbb{R}^n\setminus\{0\})
\end{equation}
 where $\nu_0$ is the L\'evy measure corresponding to $\psi$ and $\nu(t,\mathrm{d}y)$ is the L\'evy measure corresponding to $q(t,\xi)$. We refer to \cite{K1} and \cite{J2} where the condition $e^{-t\psi}\in L^1(\mathbb{R}^n)$ is related to growth conditions of $\psi$ or the doubling property. The estimate \eqref{EQ2.2} induces of course
\begin{equation}\label{EQ2.3}
\kappa_0\psi(\xi)\leq q(t,\xi)\leq\kappa_1\psi(\xi)
\end{equation}
for all $\xi\in\mathbb{R}^n$. Estimates such as \eqref{EQ2.2} or \eqref{EQ2.3} have the interpretation that corresponding pseudo-differential operators have the same continuity properties in an intrinsic scale of generalised Bessel potential spaces. Their origin is of course classical ellipticity estimates. We set
\begin{equation}
Q(t,\xi):=\int_0^tq(\tau,\xi)\mathrm{d}\tau
\end{equation}
and we find
\begin{equation}
\int_s^tq(\tau,\xi)\,\mathrm{d}\tau=Q(t,\xi)-Q(s,\xi)\geq 0
\end{equation}
and by
\begin{equation}
\hat{\mu}_{t,s}(\xi):=(2\pi)^{-\frac{n}{2}}e^{-(Q(t,\xi)-Q(s,\xi))}=(2\pi)^{-\frac{n}{2}}e^{-\int_s^tq(\tau,s)\mathrm{d}\tau}
\end{equation}
a family of probability measures $(\mu_{t,s})_{0\leq s\leq t}$ is defined. From our assumption it follows immediately that
\begin{equation}
\hat{\mu}_{s,s}(\xi)=(2\pi)^{-\frac{n}{2}}=\hat{\epsilon}_0(\xi),
\end{equation}
where $\epsilon_0$ is the Dirac measure at 0, and
\begin{equation}
\mu_{t,r}\ast\mu_{r,s}=\mu_{t,s}, \quad s\leq r\leq t.
\end{equation}
Moreover, we have
\begin{equation}
\lim_{\stackrel{s\to t}{s<t}}\hat{\mu}_{t,s}(\xi)=\hat{\epsilon}_0(\xi)
\end{equation}
and
\begin{equation}
\lim_{\stackrel{t\to s}{s<t}}\hat{\mu}_{t,s}(\xi)=\hat{\epsilon}_0(\xi)
\end{equation}
which implies the corresponding weak convergence of the measures. It follows that the family $(\mu_{t,s})_{0\leq s\leq t}$ forms the family of distributions of the increments of an additive process in law, see \cite{S1}.
\vskip5pt\noindent Moreover, from \eqref{EQ2.3} we deduce that each of the measures $\mu_{t,s}$ has a density with respect to the Lebesgue measure given by
\begin{align*}
p_{t,s}(x)&=(2\pi)^{-n}\int_{\mathbb{R}^n}e^{ix\cdot\xi}e^{-\int_s^t q(\tau,\xi)\,\mathrm{d}\tau}\mathrm{d}\xi\\
&=(2\pi)^{-n}\int_{\mathbb{R}^n}e^{ix\cdot\xi}e^{-(Q(t,\xi)-Q(s,\xi))}\mathrm{d}\xi, \ 0<s<t.
\end{align*}
As it is the inverse Fourier transform of an $L^1$-function, we have $p_{t,s}\in C_\infty(\mathbb{R}^n)$. For $t>0$ and $s=0$ we write $p_t$ for $p_{t,0}$, i.e. 
\begin{align}
p_t(x)&=(2\pi)^{-n}\int_{\mathbb{R}^n}e^{ix\cdot\xi}e^{-\int_0^tq(\tau,\xi)\mathrm{d}\tau}\mathrm{d}\xi\\
&=(2\pi)^{-n}\int_{\mathbb{R}^n}e^{ix\cdot\xi}e^{-Q(t,\xi)}\mathrm{d}\xi.\nonumber
\end{align}

\section{On Fundamental Solutions}
Let $q,Q$ and $\mu_{t,s}$ and $p_{t,s}$ be as in Section 2. On the Schwartz space $\mathcal{S}(\mathbb{R}^n)$ we may define the operators
\begin{equation}
q(t,D)u(x):=(2\pi)^{-\frac{n}{2}}\int_{\mathbb{R}^n}e^{ix\cdot\xi}q(t,\xi)\hat{u}(\xi)\,\mathrm{d}\xi
\end{equation}
as well as
\begin{equation}\label{EQ3.2}
H_{t,s}u(x):=\int_{\mathbb{R}^n}u(x-y)\mu_{t,s}(\mathrm{d}y), \ 0\leq s\leq t.
\end{equation}
Applying the convolution theorem, we obtain
\begin{align*}
(H_{t,s}u)^{\wedge}(\xi)&=(u\ast\mu)_{t,s}^{\wedge}(\xi)\\
&=(2\pi)^{\frac{n}{2}}\hat{u}(\xi)\hat{\mu}_{t,s}(\xi)\\
&=e^{-(Q(t,\xi)-Q(s,\xi))}\hat{u}(\xi),
\end{align*}
or
$$
H_{t,s}u(x)=(2\pi)^{-\frac{n}{2}}\int_{\mathbb{R}^n}e^{ix\cdot\xi}e^{-(Q(t,\xi)-Q(s,\xi))}\hat{u}(\xi)\,\mathrm{d}\xi.
$$
We want to study the operators $(H_{t,s})_{0<s<t}$ in $L^2(\mathbb{R}^n)$ and $C_\infty(\mathbb{R}^n)$. The properties of $(\mu_{t,s})_{0\leq s\leq t}$ imply immediately on $\mathcal{S}(\mathbb{R}^n)$
\begin{equation}
H_{s,s}u=u,
\end{equation}
or
\begin{equation}
H_{s,s}=id
\end{equation}
and
\begin{equation}
(H_{t,r}\circ H_{r,s})u=H_{t,r}(H_{r,s}u)=H_{t,s}u,
\end{equation}
or
\begin{equation}
H_{t,r}\circ H_{r,s}=H_{t,s}.
\end{equation}
Moreover, we have
\begin{equation}\label{EQ3.7}
\|H_{t,s}u\|_{\infty}\leq\|u\|_\infty
\end{equation}
and by Plancherel's theorem
\begin{equation}\label{EQ3.8}
\|H_{t,s}u\|_{L^2}\leq \|u\|_{L^2}.
\end{equation}
The weak convergence properties of $(\mu_{t,s})_{0<s<t}$ yield also
\begin{equation}
\lim_{\stackrel{s\to t}{s<t}}\|H_{t,s}u-u\|_\infty=\lim_{\stackrel{t\to s}{s<t}}\|H_{t,s}u-u\|_\infty=0
\end{equation}
and since by Plancherel's theorem
\begin{equation}\label{EQ3.10}
\|H_{t,s}u-u\|_0^2=\int_{\mathbb{R}^n}\left|e^{(Q(t,\xi)-Q(s,\xi))}-1\right|^2|\hat{u}(\xi)|^2\,\mathrm{d}\xi
\end{equation}
we deduce
\begin{equation}
\lim_{\stackrel{s\to t}{s<t}}\|H_{t,s}u-u\|_0=\lim_{\stackrel{t\to s}{s<t}}\|H_{t,s}u-u\|_0=0.
\end{equation}
\begin{lemma}\label{LEM3.1}
For $u\in\mathcal{S}(\mathbb{R}^n)$ and $t>s>0$ we have
\begin{equation}\label{EQ3.12}
\frac{\partial}{\partial t}H_{t,s}u(x)=-q(t,D)H_{t,s}u(x)
\end{equation}
and
\begin{equation}\label{EQ3.13}
\frac{\partial}{\partial s}H_{t,s}u(x)=-H_{t,s}(-q(s,D)u)(x).
\end{equation}
\end{lemma}
\begin{proof}
Using the definitions, we obtain for $u\in\mathcal{S}(\mathbb{R}^n)$ and $0<s<t$ that
\begin{align*}
\frac{\partial}{\partial t}H_{t,s}u(x)&=(2\pi)^{-\frac{n}{2}}\int_{\mathbb{R}^n}e^{ix\cdot\xi}\frac{\partial}{\partial t}\left(e^{-(Q(t,\xi)-Q(s,\xi))}\right)\hat{u}(\xi)\,\mathrm{d}\xi\\
&=(2\pi)^{-\frac{n}{2}}\int_{\mathbb{R}^n}e^{ix\cdot\xi}\left(-\frac{\partial}{\partial t}Q(t,\xi)\right)e^{-(Q(t,\xi)-Q(s,\xi))}\hat{u}(\xi)\,\mathrm{d}\xi\\
&=(2\pi)^{-\frac{n}{2}}\int_{\mathbb{R}^n}e^{ix\cdot\xi}(-q(t,\xi))e^{(Q(t,\xi)-Q(s,\xi))}\hat{u}(\xi)\,\mathrm{d}\xi\\
&=-q(t,D)H_{t,s}u(x),\\
\end{align*}
which proves \eqref{EQ3.12}. Further we get
\begin{align*}
\frac{\partial}{\partial s}H_{t,s}u(x)&=(2\pi)^{-\frac{n}{2}}\int_{\mathbb{R}^n}e^{ix\cdot\xi}\left(\frac{\partial}{\partial s}e^{-(Q(t,\xi)-Q(s,\xi))}\right)\hat{u}(\xi)\,\mathrm{d}\xi\\
&=(2\pi)^{-\frac{n}{2}}\int_{\mathbb{R}^n}e^{ix\cdot\xi}e^{-(Q(t,\xi)-Q(s,\xi))}\left(\frac{\partial}{\partial s}Q(s,\xi)\right)\hat{u}(\xi)\,\mathrm{d}\xi\\
&=(2\pi)^{-\frac{n}{2}}\int_{\mathbb{R}^n}e^{ix\cdot\xi}e^{-(Q(t,\xi)-Q(s,\xi))}q(s,\xi)\hat{u}(\xi)\,\mathrm{d}\xi\\
&=-H_{t,s}(-q(s,D)u)(x),
\end{align*}
and the lemma is proved.
\end{proof}
\noindent By \eqref{EQ3.7} we can extend $H_{t,s}$ continuously to $C_\infty(\mathbb{R}^n)$ and by \eqref{EQ3.8} we can extend $H_{t,s}$ continuously to $L^2(\mathbb{R}^n)$. In each case, we will use $H_{t,s}$ to denote the extension. It is clear that \eqref{EQ3.7} and \eqref{EQ3.8}-\eqref{EQ3.10} also hold for the extension. More care is needed for extending Lemma 3.1 to $C_\infty(\mathbb{R}^n)$. The $L^2$-case is however not too difficult to deal with. Using $\psi$ from \eqref{EQ2.3}, we introduce the space
\begin{equation}
H^{\psi,2}(\mathbb{R}^n):=\{v\in L^2(\mathbb{R}^n)\,|\, \|u\|_{\psi,2}<\infty\}
\end{equation}
where
\begin{equation}
\|v\|_{\psi,2}^2=\int_{\mathbb{R}^n}(1+\psi(\xi))^2|\hat{v}(\xi)|^2\,\mathrm{d}\xi.
\end{equation}
The uniformity of estimate \eqref{EQ2.3} with respect to $t$ implies that the operator $(-q(t,D),H^{\psi,2}(\mathbb{R}^n))$ is a closed $L^2$-operator and that \eqref{EQ3.12} as well as \eqref{EQ3.13} hold as equations in $L^2(\mathbb{R}^n)$. In order to interpret this observation, we recall, see \cite{T}:

\begin{definition}
Let $(X,\|\cdot\|_X)$ be a Banach space. Suppose that for every $t>0$ an operator $(A(t),D(A(t)))$ on $X$ is given which for each $t_0>0$ fixed generates a strongly continuous contraction semi-group on $X$. Suppose that $D(A(t))$ is independent of $t$. We call a strongly continuous family $(U(t,s))_{0\leq s\leq t}$, $0\leq s\leq t$, $0\leq t\leq T$, of bounded operators $U(t,s):X\to X$ an $\mathbf{X-}$\textbf{fundamental solution} to the initial value problem
\begin{equation}
\frac{\partial u(t)}{\partial t}=A(t)u(t)=f(t), \ \ 0\leq t\leq T
\end{equation}
and 
\begin{equation}
u(0)=u_0,
\end{equation}
where $u_0\in X$, $u(\cdot)\in D(A(t))$, $f\in C([0,T];X)$, if we have
\begin{align}
&U(t,r)U(r,s)=U(t,s) \ \ \text{for} \ \ 0\leq s\leq r\leq t\leq T;\label{EQ3.18}\\
&U(s,s)=id \ \ \text{for} \ \ 0\leq s\leq T;\\
&\frac{\partial}{\partial t}U(t,s)=-A(t)U(t,s), \ \ 0\leq s\leq t\leq T;
\end{align}
and
\begin{equation}\label{EQ3.21}
\frac{\partial}{\partial s}U(t,s)=U(t,s)A(s), \ \ 0\leq s\leq t\leq T.
\end{equation}
\end{definition}
\noindent Thus, we have by the calculations from the proof of Lemma \ref{LEM3.1},
\begin{theorem}\label{THM3.3}
The family $(H_{t,s})_{0\leq s\leq t\leq T}$ is an $L^2$-fundamental solution to the problem
\begin{equation}
\frac{\partial}{\partial t}u(t,x)+q(t,D)u(t,x)=f(t,x), \ u(0,x)=u_0(x),
\end{equation}
where the domain of $q(t,D)$ is $H^{\psi,2}(\mathbb{R}^n)$, and $\psi$ is taken from \eqref{EQ2.3}.
\end{theorem}
\noindent The situation for $C_\infty(\mathbb{R}^n)$ is (as we must expect) more complicated. Using the L\'evy measure $\nu(t,\mathrm{d}y)$ and representation \eqref{EQ3.2}, we can prove that $C_\infty^2(\mathbb{R}^n)\cap C_\infty(\mathbb{R}^n)$ will be in the domain of the generator of the Feller semi-group $(T_t^{q(t_0,\cdot)})_{t\geq 0}$ associated with $q(t_0,\cdot)$ and that this domain is independent of $t$. Then Theorem \ref{THM3.3} can be extended to the case where $L^2(\mathbb{R}^n)$ is replaced by $C_\infty(\mathbb{R}^n)$. For our purposes, it is sufficient to note that by \eqref{EQ2.3} the domain of the generator of $(T_t^{q(t_0,\cdot)})_{t\geq 0}$ is independent of $t_0$ and that $\mathcal{S}(\mathbb{R}^n)$ is a subspace of the domain on which \eqref{EQ3.18}-\eqref{EQ3.21} hold.

\section{On Adjoint Distributions}
We use the notation and assumptions of the previous sections and introduce the probability measures
\begin{equation}\label{EQ4.1}
\rho_t:=\tilde{\rho}(\cdot)\lambda^{(n)}:=\frac{e^{-Q(t,\cdot)}}{(2\pi)^{\frac{n}{2}}p_t(0)}, \ \ t>0.
\end{equation}
From \eqref{EQ4.1} we obtain
\begin{equation}
\hat{\rho}_t(y)=\frac{p_t(y)}{p_t(0)}.
\end{equation}
Our assumptions on $q(t,\cdot)$, in particular, \eqref{EQ2.2} and \eqref{EQ2.3} imply for every $\delta>0$ that
\begin{equation}
\inf_{|\xi|\geq \delta} q(\tau,\xi)\geq \kappa_0\inf_{|\xi|\geq \delta}\psi(\xi)=:M_\delta>0,
\end{equation}
where the last estimate follows from the fact that $\psi(\xi)>0$ for $\xi\neq0$.
\\Following the proof of Lemma 5.6 in \cite{K1}, we find
$$
\int_{|\xi|\geq \delta}e^{-Q(t,\xi)}\,\mathrm{d}\xi=\int_{|\xi|\geq \delta}e^{-\int_0^t q(\tau,\xi)\,\mathrm{d}\tau}\mathrm{d}\xi\leq\int_{|\xi|\geq \delta}e^{-t\kappa_0\psi(\xi)}\mathrm{d}\xi
$$
or for $0<t_0<t$
\begin{equation}\label{EQ4.4}
\int_{|\xi|\geq\delta}e^{-Q(t,\xi)}\,\mathrm{d}\xi\leq e^{-(t-t_0)M_\delta}\int_{|\xi|\geq \delta}e^{-t_0\kappa_0\psi(\xi)}\,\mathrm{d}\xi.
\end{equation}
Since
\begin{equation}
\psi(\xi)\leq C_R^\psi |\xi|^2+a_R^\psi,
\end{equation}
where $C_R^\psi\asymp\int_{|y|\leq R}|y|^2\nu(\mathrm{d}y)$ and $a_R^\psi\asymp\nu_0(B_R^\complement(0))$ it follows that 
\begin{align}
\int_{\mathbb{R}^n}e^{-Q(t,\xi)}\mathrm{d}\xi&=\int_{\mathbb{R}^n}e^{-\int_0^tq(\tau,\xi)\mathrm{d}\tau}\mathrm{d}\xi\geq\int_{\mathbb{R}^n}e^{-t\kappa_1\psi(\xi)}\mathrm{d}\xi\nonumber\\
&\geq \int_{\mathbb{R}^n}e^{-t\kappa_1C_R^\psi|\xi|^2\,\mathrm{d}\xi}e^{-ta_R^\psi},\label{EQ4.6}
\end{align}
here $a\asymp b$ means that $0<\gamma_1\leq\frac{b}{a}\leq\gamma_2$. Combining \eqref{EQ4.4} with \eqref{EQ4.6} we obtain, compare with \cite{K1},
\begin{align*}
\frac{\int_{|\xi|>\delta}e^{-Q(t,\xi)}\mathrm{d}\xi}{(2\pi)^{-\frac{n}{2}}p_t(0)}&\leq \frac{e^{-(t-t_0)M_\delta}\int_{|\xi|>\delta}e^{-t_0\kappa_0\psi(\xi)}\mathrm{d}\xi}{(2\pi)^{-\frac{n}{2}}\int_{\mathbb{R}^n}e^{-t\kappa_1C_R^\psi|\xi|^2}\mathrm{d}\xi e^{-ta_R^\psi}}\\
&=\frac{e^{-(t-t_0)M_\delta}\int_{|\xi|>\delta}e^{-t_0\kappa_0\psi(\xi)}\mathrm{d}\xi}{(2\pi)^{-\frac{n}{2}}t^{-\frac{n}{2}}e^{-ta_R^\psi}\int_{\mathbb{R}^n}e^{-\kappa_1C_R^\psi|\eta|^2}\mathrm{d}\eta}\\
&=t^{\frac{n}{2}}e^{-t(M_\delta-a_R^\psi)}e^{t_0\mu_\delta}\frac{\int_{|\xi|>\delta}e^{-t_0\kappa_0\psi(\xi)}\mathrm{d}\xi}{(2\pi)^{\frac{n}{2}}\int_{\mathbb{R}^n}e^{-\kappa_1C_R^\psi|\eta|^2}\mathrm{d}\eta}.
\end{align*}
We may choose for a given $\delta>0$ the value of $R>0$ such that $M_\delta>a_R^\psi$ and we have proved
\begin{lemma}\label{LEM4.1}
For $\delta>0$ and $t>0$, we have
\begin{equation}
\lim_{t\to\infty}\frac{\int_{|\xi|>\delta}e^{-Q(t,\xi)}\mathrm{d}\xi}{(2\pi)^{-\frac{n}{2}}\int_{\mathbb{R}^n}e^{-Q(t,\xi)}\mathrm{d}\xi}=0.
\end{equation}
\end{lemma}
\vskip10pt\noindent Now, for $t>0$ and $\eta\in\mathbb{R}^n$ it follows for $u\in C_\infty(\mathbb{R}^n)$ that
\begin{align*}
\lefteqn{\left|\int_{\mathbb{R}^n}\tilde{\rho}_t(\xi)(u(\eta-\xi)-u(\eta)\,\mathrm{d}\xi\right|}\hskip2cm\\
&\leq \int_{|\xi|\leq\delta}\tilde{\rho}_t(\xi)|u(\eta-\xi)-u(\eta)|\,\mathrm{d}\xi+2\int_{|\xi|>\delta}\tilde{\rho}_t(\xi)\,\mathrm{d}\xi\|u\|_\infty\\
&\leq\sup_{|\xi|\leq\delta}|u(\eta-\xi)-u(\eta)|+2\int_{|\xi|\geq\delta}\tilde{\rho}_t(\xi)\,\mathrm{d}\xi\|u\|_\infty
\end{align*}
and Lemma \ref{LEM4.1} now implies
\begin{lemma}\label{LEM4.2}
\noindent For $u\in C_\infty(\mathbb{R}^n)$ we have
\begin{equation}
\lim_{t\to\infty}\int_{\mathbb{R}^n}\tilde{\rho}_t(\xi)u(\eta-\xi)\,\mathrm{d}\xi=u(\eta).
\end{equation}
\end{lemma}
For $u\in\mathcal{S}(\mathbb{R}^n)$ we define
\begin{equation}
(S_t u)(x):=(\rho_{\frac1t}\ast u)(x)=(2\pi)^{-\frac{n}{2}}\int_{\mathbb{R}^n}e^{ix\xi}(\rho_\frac1t\ast u)^\wedge(\xi)\,\mathrm{d}\xi.
\end{equation}
Since by the convolution theorem
\begin{equation}
(\rho_\frac1t\ast u)^\wedge(\xi)=(2\pi)^{\frac{n}{2}}\hat{\rho}_\frac1t(\xi)\hat{u}(\xi)
\end{equation}
and $\hat{\rho}_\frac1t(\xi)=\frac{p_\frac1t(\xi)}{p_\frac1t(0)}$ we get (at least on $\mathcal{S}(\mathbb{R}^n)$)
\begin{equation}
(S_tu)(x)=(2\pi)^{-\frac{n}{2}}\int_{\mathbb{R}^n}e^{ix\cdot\xi}\frac{p_\frac1t(\xi)}{p_\frac1t(0)}\hat{u}(\xi)\,\mathrm{d}\xi.
\end{equation}
With
\begin{equation}
\sigma_t(\xi):=\frac{p_\frac1t(\xi)}{p_\frac1t(0)}
\end{equation}
we have
\begin{equation}
(S_tu)(x)=(2\pi)^{-\frac{n}{2}}\int_{\mathbb{R}^n}e^{ix\cdot\xi}\sigma_t(\xi)\hat{u}(\xi)\,\mathrm{d}\xi.
\end{equation}
Since $p_{\frac1t}(\xi)\leq p_{\frac1t}(0)$ for $t>0$, our construction yields
\begin{equation}
\|S_tu\|_\infty\leq\|u\|_\infty
\end{equation}
as well as
\begin{equation}
\|S_tu\|_{L^2}\leq\|u\|_{L^2},
\end{equation}
and from Lemma \ref{LEM4.2} and its proof we now deduce
\begin{equation}
\lim_{t\to0}\|S_tu-u\|_\infty=\lim_{t\to\infty}\|S_tv-v\|_{L^2}=0
\end{equation}
for all $u\in C_\infty(\mathbb{R}^n)$ and $v\in L^2(\mathbb{R}^n)$, respectively. We note further that
\begin{align*}
\frac{\partial}{\partial t}\sigma_t(\xi)&=\frac{\partial}{\partial t}\frac{p_\frac1t(\xi)}{p_\frac1t(0)}\\
&=\sigma_t(\xi)\frac{\partial}{\partial t}\ln\sigma_t(\xi)
\end{align*}
we set
\begin{equation}\label{EQ4.17}
A(t,\xi):=-\frac{\partial}{\partial t}\ln\sigma_t(\xi).
\end{equation}
and consider on $\mathcal{S}(\mathbb{R}^n)$ the operator
\begin{equation}
A(t,D)u(x):=(2\pi)^{-\frac{n}{2}}\int_{\mathbb{R}^n}e^{ix\cdot\xi}A(t,\xi)\hat{u}(\xi)\,\mathrm{d}\xi.
\end{equation}
We first observe that
\begin{align*}
\frac{\partial}{\partial t}S_tu(x)&=\frac{\partial}{\partial t}\left((2\pi)^{-\frac{n}{2}}\int_{\mathbb{R}^n}e^{ix\cdot\xi}\frac{p_{\frac1t}(\xi)}{p_{\frac1t}(0)}\hat{u}(\xi)\,\mathrm{d}\xi\right)\\
&=\frac{\partial}{\partial t}\left((2\pi)^{-\frac{n}{2}}\int_{\mathbb{R}^n}e^{ix\cdot\xi}\sigma_t(\xi)\hat{u}(\xi)\,\mathrm{d}\xi\right)\\
&=(2\pi)^{-\frac{n}{2}}\int_{\mathbb{R}^n}e^{ix\cdot\xi}\frac{\partial}{\partial t}(\sigma_t(\xi))\hat{u}(\xi)\,\mathrm{d}\xi\\
&-(2\pi)^{-\frac{n}{2}}\int_{\mathbb{R}^n}e^{ix\cdot\xi}\left(\frac{\partial}{\partial t}\ln\sigma_t(\xi)\right)\sigma_t(\xi)\hat{u}(\xi)\,\mathrm{d}\xi\\
&=-A(t,D)(S_tu)(x),
\end{align*}
or
\begin{equation}
\frac{\partial}{\partial t}S_tu+A(t,D)S_tu=0.
\end{equation}
We now introduce the family of operators $V(t,s)$, $0<s<t$, by
\begin{equation}
(V(t,s)u)^\wedge(\xi)=e^{-\int_s^tA(\tau,\xi)\,\mathrm{d}\tau}\hat{u}(\xi), \ \ u\in \mathcal{S}(\mathbb{R}^n).
\end{equation}
 The condition $A(t,\xi)\geq 0$ will already lead to a satisfactory $L^2$-theory for the operator $V(t,s)$, $0<s<t$. However, since we eventually want to investigate adjoint processes we add here:
\vskip10pt\noindent \textbf{Basic Assumption I.} We assume that for all $t>0$ the function $\xi\mapsto A(t,\xi)$ is a real continuous negative definite function.
\vskip10pt\noindent This is clearly a substantial and restrictive assumption and it is open to characterise those symbols $q(\tau,\xi)$ which eventually will lead to a symbol $A(t,\xi)$ satisfying this assumption. Non-trivial examples will be provided in Section 6.
\vskip10pt\noindent Under Basic Assumption I, it follows that $e^{-\int_s^tA(\tau,\xi)\,\mathrm{d}\tau}$ is a positive definite function in the sense of Bochner, hence by
\begin{equation}\label{EQ4.21}
\hat{\gamma}_{t,s}(\xi):=(2\pi)^{-\frac{n}{2}}e^{-\int_s^tA(\tau,\xi)\,\mathrm{d}\tau}
\end{equation}
a family of probability measures $\gamma_{t,s}, 0<s<t$ is defined. From \eqref{EQ4.21} we deduce immediately
\begin{align}
\gamma_{s,s}=\epsilon_0, \ \ 0\leq s;\label{EQ4.22}\\
\gamma_{t,r}\ast\gamma_{r,s}=\gamma_{t,s}, \ \ 0<s<r<t;\\
\gamma_{t,s}\to\epsilon_0 \ \text{ weakly for } \ s\to t, s<t;\\
\gamma_{t,s}\to\epsilon_0 \ \text{ weakly for } \ t\to s, s<t. \label{EQ4.25}
\end{align}
Following \cite{S1}, Theorem 9.7, we can associate with $(\gamma_{t,s})_{0<s<t<\infty}$ a canonical additive process in law with state space $\mathbb{R}^n$. Thus we have proved
\begin{theorem}
Let $q:[0,\infty):\mathbb{R}^n\to\mathbb{R}$ and $\psi:\mathbb{R}^n\to\mathbb{R}$ satisfying the assumptions of Section 2 and suppose that $A(t,\xi)$ defined by \eqref{EQ4.17} fulfils Basic Assumption I. Then we can associate with $q(t,\xi)$ an additive process in law $(Y_t)_{t\geq 0}$ and with $A(t,\xi)$ we can associate an additive process in law $(X_t)_{t\geq0}$. The distributions of the increments are given by
\begin{equation}
P_{Y_t-Y_s}=\mu_{t,s}
\end{equation}
and
\begin{equation}
P_{X_t-X_s}=\gamma_{t,s}.
\end{equation}
\end{theorem}
\begin{definition}
We call $(Y_t)_{t\geq 0}$ and $(X_t)_{t\geq 0}$ a pair of \textbf{adjoint additive processes} in law.
\end{definition}
\par\noindent Using \eqref{EQ4.22}-\eqref{EQ4.25}, or directly \eqref{EQ4.21}, it is straightforward to see that we can extend $(V(t,s))_{0<s<t}$ as an $X$-fundamental solution to $-A(t,D)$ for $X\in\{C_\infty(\mathbb{R}^n, L^2(\mathbb{R}^n)\}$. However, even in the case $X=L^2(\mathbb{R}^n)$ it is not obvious how to characterise $D(A(t))$ in terms of $\psi$, one of the data characterising our construction. 

\section{Some Geometric Interpretations of the Densities}

The measures $\mu_{t,s}$ and $\gamma_{t,s}$ have densities with respect to the Lebesgue measure, indeed we have
\begin{equation}
P_{Y_t-Y_s}=\mu_{t,s}=F^{-1}\left(e^{-(Q(t,\cdot)-Q(s,\cdot))}\right)\lambda^{(n)}=p_{t,s}(\cdot)\lambda^{(n)}
\end{equation}
and
\begin{align}
P_{X_t-X_s}=\gamma_{t,s}&=F^{-1}\left((2\pi)^{-\frac{n}{2}}e^{-\int_s^tA(\tau,\xi)\,\mathrm{d}\tau}\right)\lambda^{(n)}\nonumber\\
&=F^{-1}\left((2\pi)^{-\frac{n}{2}}e^{-\int_s^t\frac{\partial}{\partial\tau}\ln\sigma_\tau(\xi)\,\mathrm{d}\tau}\right)\lambda^{(n)}\nonumber\\
&=F^{-1}\left((2\pi)^{-\frac{n}{2}}e^{\ln\sigma_t(\cdot)-\ln\sigma_s(\cdot)}\right)\lambda^{(n)}\nonumber\\
&=F^{-1}\left((2\pi)^{-\frac{n}{2}}\frac{p_\frac1t(\cdot)}{p_\frac1t(0)}\cdot\frac{p_\frac1s(0)}{p_\frac1s(\cdot)}\right)\lambda^{(n)}\label{EQ5.2}
\end{align}
Some care is needed with \eqref{EQ5.2}. Since by Basic Assumption I $\int_s^tA(\tau,\xi)\,\mathrm{d}\tau$ is a continuous negative definite function, it follows that $\int_s^tA(\tau,\xi)\,\mathrm{d}\tau\geq 0$ and at least in the sense of $\mathcal{S}'(\mathbb{R}^n)$ we can calculate the inverse Fourier transform of $e^{-\int_s^tA(\tau,\xi)\,\mathrm{d}\tau}$. In fact we know more, namely that $e^{-\int_s^t A(\tau,\xi)\,\mathrm{d}\tau}$  is a positive definite function. Thus \eqref{EQ5.2} is justified. However, while we can guarantee that $\frac{p_\frac1t(\cdot)}{p_\frac1t(0)}$ belongs to $L^1(\mathbb{R}^n)$, we cannot a priori guarantee  that  $\frac{p_\frac1s(0)}{p_\frac1s(\cdot)}$ belongs to $\mathcal{S}'(\mathbb{R}^n)$, and we cannot a priori apply the convolution theorem to \eqref{EQ5.2}. 
\\For the case $s=0$, however, we obtain
$$
\mu_t:=P_{Y_t-Y_0}=\mu_{t,0}=F^{-1}\left(e^{-Q(t,\cdot)}\right)=p_t(\cdot)\lambda^{(n)}
$$
and using a consequence of Lemma \ref{LEM4.2}, namely that $\lim_{s\to0}\sigma_\frac1s=1,$ we obtain
\begin{align}
\gamma_t:=P_{X_t-X_0}=\gamma_{t,0}&=F^{-1}\left((2\pi)^{-\frac{n}{2}}e^{-\int_0^tA(\tau,\xi)\,\mathrm{d}\tau}\right)\lambda^{(n)}\\
&=F^{-1}\left((2\pi)^{-\frac{n}{2}}e^{\ln\sigma_t(\cdot)}\right)\lambda^{(n)}=\frac{1}{(2\pi)^{\frac{n}{2}}}F^{-1}(\sigma_t(\cdot))\lambda^{(n)}\nonumber\\
&=\frac{1}{(2\pi)^{\frac{n}{2}}}F^{-1}\left(\frac{p_\frac1t(\xi)}{p_\frac1t(0)}\right)\lambda^{(n)},\nonumber
\end{align}
i.e. 
\begin{equation}
\gamma_t=\Phi_t(\cdot)\lambda^{(n)}:=\frac{e^{-Q(\frac1t,\cdot)}}{(2\pi)^{\frac{n}{2}}p_{\frac1t}(0)}\lambda^{(n)}.
\end{equation}
Our aim is to give geometric interpretations for $p_t$ as well as for $\Phi_t$ and for this we follow closely the ideas of \cite{B4} which are based on \cite{J2}. For this we add:
\vskip10pt\noindent\textbf{Basic Assumption II.} For the continuous negative definite function $\psi$ from \eqref{EQ2.3} by $\mathrm{d}_\psi(\xi,\eta):=\sqrt{\psi(\xi-\eta)}$ a metric is defined on $\mathbb{R}^n$ which generates the Euclidean topology. Moreover, we assume that $(\mathbb{R}^n,\mathrm{d}_\psi,\lambda^{(n)})$ has the \textbf{volume doubling property}, i.e. 
\begin{equation}
\lambda^{(n)}(B^{\mathrm{d}_\psi}(x,2r))\leq c_0\lambda^{(n)}(B^{\mathrm{d}_\psi}(x,r))
\end{equation} 
for all $x\in\mathbb{R}^n$ and $r>0$ where $B^{\mathrm{d}_\psi}(x,r)=\{y\in\mathbb{R}^n|\mathrm{d}_\psi(x,y)<r\}$ is the open ball with respect to $\mathrm{d}_\psi$ with centre $x$ and radius $r$.
\vskip10pt\noindent Note that if $\psi:\mathbb{R}^n\to\mathbb{R}$ is a continuous negative definite function such that $\psi(\xi)=0$ if and only if $\xi=0$, then $\mathrm{d}_\psi$ is always a metric on $\mathbb{R}^n$. In \cite{J2}, in particular Lemma 3.2, conditions are proved for $\mathrm{d}_\psi$ to generate the Euclidean topology, and the volume doubling property of $\mathrm{d}_\psi$ is discussed in more detail.
\vskip10pt\noindent Since in \eqref{EQ2.3} we can replace $\psi$ by $q(t_0,\cdot)$ for a fixed $t_0>0$ (with a change of the constants $\kappa_0$ and $\kappa_1$), we can transfer the results of Section 4 in \cite{B4}. Thus, it follows that under Basic Assumption II with
\begin{equation}
Q_{t,s}(\xi)=\int_s^tq(\tau,\xi)\,\mathrm{d}\tau
\end{equation}
a new metric is given by
\begin{equation}
d_{Q_{t,s}}(\xi,\eta):=Q_{t,s}^{\frac12}(\xi-\eta), \ \ 0\leq s<t
\end{equation}
and this metric generates the Euclidean topology on $\mathbb{R}^n$ and has the volume doubling property. This applies, in particular, to $d_{Q_{t,0}}$. The proof of Theorem 4.1 in \cite{B4}, compare also with Theorem 4.1 in \cite{J2}, yields under Basic Assumption I and Basic Assumption II that
\begin{equation}
p_{t,s}(0)=(2\pi)^{-n}\int_{\mathbb{R}^n}\lambda^{(n)}(B^{\mathrm{d}_{Q_{t,s}}}(0,\sqrt{r}))e^{-r}\,\mathrm{d}r
\end{equation}
and using the volume doubling property, as well as \eqref{EQ2.3}, we get
\begin{equation}
p_{t,s}(0)\asymp\lambda^{(n)}(B^{\mathrm{d}_{Q_{t,s}}}(0,\sqrt{\tfrac{\kappa_1}{\kappa_0}})).
\end{equation}
We now consider the case $s=0$ and write $p_t=p_{t,0}$ etc. It follows that
\begin{align*}
p_t(x)&=p_t(0)\frac{p_t(x)}{p_t(0)}=p_t(0)e^{\ln\left(\frac{p_t(x)}{p_t(0)}\right)}\\
&=p_t(0)e^{-(-\ln\sigma_{\frac{1}{t}}(x))}=p_t(0)e^{-((-\ln\sigma_\frac1t(x))^\frac12)^2}
\end{align*}
and by our assumptions, for $t>0$ fixed, a metric is given by
\begin{equation}
\delta_{Q_{t,0}}(x,y):=(-\ln\sigma_{\frac1t}(x-y))^\frac12
\end{equation}
which allows us to write 
\begin{equation}
p_t(x-y)=p_t(0)e^{-\delta^2_{Q_{t,0}}(x,y)}
\end{equation}
with $p_t(0)\asymp\lambda^{(n)}(B^{\mathrm{d}_{Q_{t,0}}}(0,\sqrt{\tfrac{\kappa_1}{\kappa_0}}))$. On the other hand we have
\begin{equation}
\Phi_t(x)=\Phi_t(0)\frac{\Phi_t(x)}{\Phi_t(0)}=\Phi_t(0)e^{-Q_{1/ t,0}(x,0)}
\end{equation}
or
\begin{equation}
\Phi_t(x-y)=\Phi_t(0)e^{-\mathrm{d}^2_{Q_{1/ t,0}}(x,y)}.
\end{equation}
For $\Phi_t(0)$ we have
\begin{equation}
\Phi_t(0)=(2\pi)^{-n}\int_{\mathbb{R}^n}e^{-\int_0^tA(\tau,\xi)\,\mathrm{d}\tau}\mathrm{d}\xi,
\end{equation}
but
\begin{equation}
\ln\sigma_t(\xi)=-\int_0^tA(\tau,\xi)\,\mathrm{d}\tau.
\end{equation}
It follows from the definition of $\sigma_t$ that we can write
\begin{equation}
\Phi_t(0)=(2\pi)^{-n}\int_{\mathbb{R}^n}e^{-(-\ln\sigma_t(\xi))}\mathrm{d}\xi
\end{equation}
and $-\ln\sigma_t$ is the square of a metric, namely $-\ln\sigma_t=\delta^2_{Q_{1/t,0}}$. We can now use the arguments in \cite{B4} to obtain
\begin{equation}
\Phi_t(0)=(2\pi)^{-n}\int_0^\infty\lambda^{(n)}(B^{\delta_{Q_{1/t,0}}}(0,\sqrt{r}))e^{-r}\,\mathrm{d}r
\end{equation}
and eventually we have the dual formulae
\begin{equation}\label{EQ5.18}
p_t(x-y)=(2\pi)^{-n}\int_{\mathbb{R}^n}\lambda^{(n)}(B^{\mathrm{d}_{Q_{t,0}}}(0,\sqrt{r}))e^{-r}\,\mathrm{d}r\, e^{-\delta^2_{Q_{t,0}}(x,y)}
\end{equation}
and
\begin{equation}\label{EQ5.19}
\Phi_t(x-y)=(2\pi)^{-n}\int_{\mathbb{R}^n}\lambda^{(n)}(B^{\delta_{Q_{1/t,0}}}(0,\sqrt{r}))e^{-r}\,\mathrm{d}r\, e^{-\mathrm{d}^2_{Q_{1/t,0}}(x,y)}.
\end{equation}
Thus, under our assumptions of Section 2, Basic Assumptions I and II and the assumption that $p_t$ is unimodal, we obtain for the two additive processes $(Y_t)_{t\geq 0}$ generated by $-q(t,D)$ and $(X_t)_{t\geq 0}$ generated by $-A(t,D)=\left(\frac{\partial}{\partial t}\ln\sigma_t\right)(D)$ the dual formulae \eqref{EQ5.18} and \eqref{EQ5.19} for the transition densities of $Y_t$ and $X_t$ respectively. 

\section{Examples}

\begin{example}
In this example we consider the case where $Q(t,\xi)=h(t)|\xi|^2$, $h(t)> 0$ for $t>0$, $h(0)=0$ and for $h$ strictly increasing. We first consider the transition densities $p_{t,0}(x)$ for $t>0$,
\begin{align*}
p_{t,0}(x)&=(2\pi)^{-n}\int_{\mathbb{R}^n}e^{ix\cdot\xi}e^{-h(t)|\xi|^2}\,\mathrm{d}\xi\\
&=\frac{1}{(4\pi h(t))^\frac{n}{2}}e^{-\frac{|x|^2}{4h(t)}}.
\end{align*}
Now, for the adjoint process we find using the fact that $h(1/t)\geq 0$ and that $t\mapsto h(1/t)$ is strictly decreasing that,
\begin{align*}
\Phi_t(x)&=(2\pi)^{-n}\int_{\mathbb{R}^n}e^{ix\cdot\xi}\frac{p_{\frac1t}(\xi)}{p_{\frac1t}(0)}\,\mathrm{d}\xi\\
&=(2\pi)^{-n}\int_{\mathbb{R}^n}e^{ix\cdot\xi}e^{\frac{-|\xi|^2}{4h(1/t)}}\,\mathrm{d}\xi\\
&=\pi^{-\frac{n}{2}}(h(1/t))^{\frac{n}{2}}e^{-|x|^2h(1/t)}.
\end{align*}
\end{example}
\begin{example}
We next consider the case where $Q(t,\xi)=h(t)|\xi|$, again where $h(t)> 0$ for $t>0$, $h(0)=0$, $h$ is strictly increasing. The transition densities for $t>0$ are given by,
\begin{align*}
p_{t,0}(x)&=(2\pi)^{-n}\int_{\mathbb{R}^n}e^{ix\cdot\xi}e^{-h(t)|\xi|}\,\mathrm{d}\xi\\
&=(\pi)^{\frac{-n-1}{2}}\Gamma\left(\frac{n+1}{2}\right)\frac{h(t)}{((h(t))^2+|\frac{x}{h(t)}|^2)^\frac{n+1}{2}}
\end{align*}

\vskip10pt\noindent Then for the adjoint we get,
\begin{align*}
\Phi_t(x)&=(2\pi)^{-n}\int_{\mathbb{R}^n}e^{ix\cdot\xi}\frac{p_{\frac1t}(\xi)}{p_{\frac1t}(0)}\,\mathrm{d}\xi\\
&=(2\pi)^{-n}\int_{\mathbb{R}^n}e^{ix\cdot\xi}\frac{(h(1/t))^{n+1}}{((h(1/t))^2+|\frac{\xi}{h(1/t)}|^2)^{\frac{n+1}{2}}}\,\mathrm{d}\xi\\
&=(2\pi)^{-\frac{n}{2}}F^{-1}\left(\frac{(h(1/t))^{n+1}}{((h(1/t))^2+|\frac{\xi}{h(1/t)}|^2)^{\frac{n+1}{2}}}\right)\\
&=\frac{2^{-\frac{n}{2}}(2\pi)^{-\frac{n}{2}}\sqrt{\pi}(h(1/t))^n}{\Gamma(\frac{n+1}{2})}e^{-h(1/t)|x|}.
\end{align*}

\end{example}

\begin{example}
Here we consider the case where $\xi$ belongs to $\mathbb{R}$, i.e. $n=1$, and $Q(t,\xi)=h(t)\ln\cosh\xi$, $h(t)> 0$ for $t>0$, $h(0)=0$ and for $h$ strictly increasing. The transition densities for $t>0$ are given by,
\begin{align*}
p_{t,0}(x)&=(2\pi)^{-n}\int_{\mathbb{R}^n}e^{ix\cdot\xi}e^{-h(t)\ln\cosh\xi}\,\mathrm{d}\xi\\
&=\frac{1}{2\pi}\int_{\mathbb{R}}e^{ix\cdot\xi}\left(\frac{1}{\cosh\xi}\right)^{h(t)}\,\mathrm{d}\xi\\
&=\frac{1}{2\pi}\int_{\mathbb{R}}e^{ix\cdot\xi}\frac{2^{h(t)}e^{-h(t)\xi}}{(1+e^{-2\xi})^{h(t)}}\,\mathrm{d}\xi\\
&=\frac{1}{2\pi}2^{h(t)-1}\int_{\mathbb{R}}\frac{2e^{-2q(t,x)\xi}}{(1+e^{-2\xi})^{p(t,x)+q(t,x)}}\,\mathrm{d}\xi,\\
\end{align*}
where
$$
q(x,t)=\frac{h(t)-ix}{2},\quad\quad p(x,t)=\frac{h(t)+ix}{2}
$$
and 
$$
p+q=h(t).
$$
Then,
\begin{align*}
p_{t,0}(x)&=\frac{1}{2\pi}2^{h(t)-1}\int_{\mathbb{R}}\frac{2(e^{-2\xi})^q}{(1+e^{-2\xi})^{p+q}}\\
&=\frac{1}{2\pi}2^{h(t)-1}\int_0^1 u^{p-1}(1-u)^{q-1}\,\mathrm{d}u\\
&=\frac{1}{2\pi}2^{h(t)-1}B(p,q)\\
&=\frac{1}{2\pi}2^{h(t)-1}B\left(\frac{h(t)+ix}{2},\frac{h(t)-ix}{2}\right)\\
&=\frac{2^{h(t)-2}}{\pi}\left|\Gamma\left(\frac{h(t)+ix}{2}\right)\right|^2.
\end{align*}
In summary,
\begin{align*}
p_{t,0}(x)&=\frac{2^{h(t)-2}}{\pi}\left|\Gamma\left(\frac{h(t)+ix}{2}\right)\right|^2,\\
p_{t,0}(0)&=\frac{2^{h(t)-2}}{\pi}\left|\Gamma\left(\frac{h(t)}{2}\right)\right|^2
\end{align*}
and
$$
\delta^2_t(x,0)=-\ln\left|\frac{\Gamma\left(\frac{h(t)+ix}{2}\right)}{\Gamma\left(\frac{h(t)}{2}\right)}\right|^2=\sum_{j=1}^\infty\ln\left(1+\frac{x^2}{(h(t)+2j)^2}\right).
$$
Our calculation made use of the one in \cite{P} where the case $q(\xi)=\ln\cosh\xi$ was treated. Further, we note that $A(t,\xi):=\sum_{j=1}^\infty\ln\left(1+\frac{x^2}{(h(1/t)+2j)^2}\right)$ fulfils our basic assumptions for $t>0$.
\end{example}

\begin{remark}
We may also combine the previous examples to form new examples, for example, we could consider
$$
Q(t,\xi,\eta)=h_1(t)|\xi|^2+h_2(t)|\eta|,
$$
where $h_i(t)> 0$ for $t>0$, $h_i(0)=0$ and for $h_i$ strictly increasing, $i=1,2$.

\end{remark}

\begin{remark}
In the case of a L\'evy process, the symbol, i.e. the characteristic exponent, can be used to obtain results with direct probabilistic interpretations, e.g. estimates for passage times. Results of this type had been extended to Feller processes generated by pseudo-differential operators with state space dependent symbols, see R. Schilling \cite{S2A}. In \cite{J3} it was pointed out that with the help of the metric $d_\psi(\xi,\eta)=\psi^\frac12(\xi-\eta)$ these results admit a geometric interpretation. For additive processes we are not aware of explicit results of this type, however by a standard procedure we can consider additive processes with state space $\mathbb{R}^n$ as time-homogeneous Markov processes with state space $\mathbb{R}^{n+1}$, see for example in the context of pseudo-differential operators the work \cite{B2}. Hence a transfer obtained for L\'evy processes to certain additive processes should be possible, but we do not want to follow up this idea here. 

\end{remark}

\vskip10pt\noindent E-mail addresses:
\vskip10pt N.Jacob@Swansea.ac.uk
\vskip5pt K.Evans@Swansea.ac.uk

\end{document}